\documentclass[11pt]{amsart}
\usepackage{amsfonts}
\usepackage{amsmath}
\usepackage{amssymb}
\usepackage{graphicx}
\usepackage{amsthm,amscd}
\usepackage{calc}
\usepackage{amsthm}

\newtheorem{theorem}{Theorem}[section]

\newtheorem{corollary}{Corollary}[section]

\newtheorem{lemma}{Lemma}[section]

\newtheorem{remark}{Remark}

\numberwithin{equation}{section}
\theoremstyle{remark}

\renewcommand{\bar}{\overline}
\newcommand{\eps}{\varepsilon}
\newcommand{\pa}{\partial}
\renewcommand{\phi}{\varphi}
\newcommand{\wt}{\widetilde}

\newcommand{\ka}{K\"ahler }

\newcommand{\C}{{\mathbb C}}
\newcommand{\Z}{{\mathbb Z}}
\newcommand{\R}{{\mathbb R}}

\newcommand{\T}{{\mathcal T}}

\newcommand{\XX}{{\mathfrak X}}
\newcommand{\XJ}{{\mathfrak J}}

\newcommand{\ke}{K\"ahler-Einstein }

\newcommand{\tei}{Teichm\"uller }
\newcommand{\h}{{\mathbb H}}

\newcommand{\g}{{\mathfrak g}}

\newcommand{\lb}{\left (}
\newcommand{\rb}{\right )}
\newcommand{\lsb}{\left [}
\newcommand{\rsb}{\right ]}
\newcommand{\lfb}{\left \{}
\newcommand{\rfb}{\right \}}

\newcommand{\dd}{\text{div}}
\newcommand{\ga}{\alpha}
\newcommand{\gb}{\beta}
\newcommand{\gm}{\gamma}
\newcommand{\gd}{\delta}
\newcommand{\rc}{\text{Ric}}

\newcommand{\wpm}{Weil-Petersson }

\title[]{On Deformations of Fano Manifolds}

\author{Huai-Dong Cao$^1$, Xiaofeng Sun$^1$, Shing-Tung Yau and Yingying Zhang$^2$}
\address{Department of Mathematics,  Lehigh University,
Bethlehem, PA 18015, USA}
\email{huc2@lehigh.edu; xis205@lehigh.edu}
\address{Department of Mathematics, Harvard University,
Cambridge, MA 02138, USA}
\email{yau@math.harvard.edu}
\address{Yau mathematical Sciences Center, Tsinghua University, Beijing, 100804, China}
\email{yingyzhang@tsinghua.edu.cn}

\thanks{$^1$Research supported in part by Simons Foundation grants.}
\thanks{$^2$Research supported by Tsinghua University Initiative Scientific Research Program}

\begin{document}

\begin{abstract}
In this paper we provide new necessary and sufficient conditions  for the existence of \ke metrics on small deformations of a Fano \ke manifold. We also show that the Weil-Petersson metric can be approximated by the Ricci curvatures of the canonical $L^2$ metrics on the direct image bundles. In addition, we describe the plurisubharmonicity of the energy functional of harmonic maps on the Kuranishi space of the deformation of compact \ke manifolds of general type.
\end{abstract}


\maketitle

\section{Introduction}\label{intro}

The existence of canonical metrics on compact complex manifolds is an important component in understanding the structure of the moduli spaces and metrics on them. Well-known examples include the \wpm metric on the moduli spaces of hyperbolic Riemann surfaces, and polarized Calabi-Yau manifolds. The classical approach to the \wpm metric is via the Kodaira-Spencer-Kuranishi theory. In this case the \wpm metric is the natural $L^2$ metric induced by the \ke metrics and the harmonic representatives of Kodaira-Spencer classes. The advantage of this classical approach is that we can define the \wpm metric pointwisely on the Kuranishi space. This is indeed the case when we study the moduli spaces of \ke manifolds of general type. Although the moduli spaces are singular in general, the complex manifold corresponding to a point in the moduli space does admit a unique \ke metric, following the work of Yau \cite{yau78}.

On the other hand,  when we study the deformation of a Fano manifold $X_0$,
although the deformation of the complex structure on $X_0$  is unobstructed, there may not be any \ke metric on such a manifold. By the recent work of Chen-Donaldson-Sun \cite{cds1, cds2, cds3} on the solution of the Yau's conjecture \cite{yauopen},  we know that the existence of \ke metrics on such manifolds is equivalent to the $K$-stability.

For a Fano \ke manifold $\lb X_0,\omega_0\rb$ with discrete holomorphic automorphism group $\text{Aut}\lb X_0\rb$,  Koiso \cite{koiso1} showed in 1983 that each small deformation of $X_0$ admits a \ke metric by using the implicit function theorem. It is more subtle when $\text{Aut}\lb X_0\rb$ is non-discrete.  In the latter case, the existence of canonical metrics such as cscK or extremal metrics were studied by Sz{\'e}kelyhidi \cite{szeke2010}, Br\"onnle \cite{ bro}, and Rollin-Simanca-Tipler \cite{rst1} in terms of the Futaki invariant or the linear stability of the action of $\text{Aut}_0\lb X_0\rb$ on the Kuranishi space of $X_0$.

In this paper, we study small deformations of Fano \ke manifolds and investigate the \wpm metric on their moduli spaces. Our first main result is the following new necessary and sufficient conditions for the existence of \ke metrics on small deformations of a Fano \ke manifold.

\begin{theorem}\label{ip10}
Let $\lb X_0,\omega_0\rb$ be a Fano \ke manifold and let $\lb \XX,B,\pi\rb$, with $X_t=\pi^{-1}(t)$, be the Kuranishi family of $X_0$ with respect to $\omega_0$. Then the following  statements are equivalent:
\begin{enumerate}
\item $X_t$ admits a \ke metric for each $t\in B$.

\smallskip
\item The dimension $h^0\lb X_t, T^{1,0}X_t\rb$ of the space of holomorphic vector fields on $X_t$ is independent of $t$ for all $t\in B$.

\smallskip
\item The automorphism group $\text{Aut}_0\lb X_t\rb$ is isomorphic to $\text{Aut}_0\lb X_0\rb$
for each $t\in B$.
\end{enumerate}
\end{theorem}

\begin{remark}
Since $h^0\lb X_t, T^{1,0}X_t\rb$ is upper semi-continuous in $t$ according to \cite{ks57}, Theorem \ref{ip10} includes Koiso's result in \cite{koiso1} as a special case.
\end{remark}

\begin{remark}
In \cite{ps2006},  Phong and Sturm introduced a stability condition preventing possible jump of the dimension of the spaces of holomorphic vector fields in the limit metric to study the convergence of the K\"ahler-Ricci flow on Fano manifolds. 
This stability condition, namely Condition (B), was further explored in Phong-Song-Sturm-Weinkove \cite{pssw1}; see also  \cite{pssw2009}.

\end{remark}

\begin{remark}
It is well known that semisimple Lie algebras are rigid. Thus the third statement of Theorem \ref{ip10} would follow from the second one directly if the Lie algebra $H^0\lb X_0, T^{1,0}X_0\rb\cong Lie\lb \text{Aut}\lb X_0\rb\rb$ is semisimple. However, $H^0\lb X_0, T^{1,0}X_0\rb$ is only reductive in general. 	
\end{remark}

Returning to the study of the \wpm metric, in \cite{fs1990} Fujiki and Schumacher defined generalized \wpm metrics on the deformation space of a family of extremal manifolds by pushing down the curvature of relative line bundles over the total space. In particular, they showed that the generalized \wpm metric for a family of \ke manifolds coincides with the classical one. Essentially,  assuming the family of \ke metrics is smooth, they showed that the curvature form of the Deligne pairing of the relative canonical bundle (or relative anti-canonical bundle) is precisely the \wpm curvature form; see also Schumacher \cite{schu1993}.  In our case, when $\lb X_0,\omega_0\rb$ is a Fano \ke manifold with non-discrete automorphism group, the existence of such smooth family of \ke metrics  is guaranteed by Theorem \ref{ip10} above, provided each fiber $X_t$ admits a \ke metric. In this case, the \wpm metric  is well-defined. Namely, it is independent of the choices of fiberwise \ke metrics. However, the $L^2$ metrics on $R^0\pi_* K_{\XX/B}^{-k}$ and their curvatures do depend on such choices in general. 

In this paper, we show that the \wpm metric  $\omega_{_{WP}}$  can be approximated by the (normalized) Ricci curvatures of the $L^2$ metrics on the direct image bundles $R^0\pi_* K_{\XX/B}^{-k}$. More precisely, we have

\begin{theorem}\label{ip30} Let  $\lb X_0,\omega_0\rb$ be a Fano \ke manifold and let $\lb \XX, B, \pi\rb$ be the Kuranishi family of $X_0$. We assume that each fiber $X_t$ admits a \ke metric. Let $\Omega=\lfb \omega_t\rfb$ be any smooth family of \ke metrics.
For each positive integer $k$, let $\rc_k=\rc\lb E_k,H_k\rb$ be the Ricci form of the $L^2$ metric $H_k\lb\Omega\rb$ on $E_k=R^0\pi_* K_{\XX/B}^{-k}$. Then
\begin{eqnarray*}
\lim_{k\to\infty}\frac{\pi^n}{k^{n+1}}\rc_k=-\omega_{_{WP}},
\end{eqnarray*}
\end{theorem}

\begin{remark}\label{kmexpan}
We note that the above approximation is natural due to the work of Fujiki-Schumacher on the curvature of the Deligne pairing, and the Knudsen-Mumford expansion of the determinant bundle of the direct image sheaf $R^0\pi_* K_{\XX/B}^{-k}$ (see Knudsen-Mumford \cite{km1976}, Zhang \cite{zhang96}, and Phong-Ross-Sturm \cite{prs08}).   In this paper, see Section \ref{wpmse}, we take a more direct approach by establishing canonical local holomorphic sections of the direct image sheaves and the deformation of \ke metrics. This leads to a systematical way to derive integral formulas for the full curvature tensors of $L^2$ metrics. While  the \ke metric on each $X_t$ is analytical in nature, the advantage of our approach is that we can approximate the \wpm metric by using algebraic metrics on each fiber.
\end{remark}

\medskip

The paper is organized as follows. In Section \ref{kudiv}, in order to give a simple criterion to check the existence of \ke metrics on small deformations of a Fano \ke manifold, we first show that, given a Fano \ke manifold $\lb X_0, \omega_0\rb$ and its Kuranishi family $\lb\XX,B,\pi\rb$ with respect to $\omega_0$, the complex structure on $X_t=\pi^{-1}(t)\subset\XX$ is compatible with the symplectic form $\omega_0$. In this case, the construction of the Kuranishi family is compatible with Donaldson's infinite dimensional GIT picture. One technical key ingredient is the equivalence of the Kuranishi gauge and the divergence gauge; see Theorem \ref{equivmain}.
Section \ref{defoke} is devoted to the proof of Theorem \ref{ip10}. In Section \ref{wpmse}, we investigate the \wpm metric and prove Theorem \ref{ip30}. An integral formula of the full curvature tensor of the $L^2$ metric $H_k\lb \Omega\rb$ on $E_k$ is also derived. In addition, we obtain the deformation formulas for the K\"ahler form $\omega_t$ and the volume form $V_t$ on $X_t$, respectively, for each $t\in B$.

Finally, in the last section, we describe the plurisubharmonicity of the energy functionals of harmonic maps on the Kuranishi spaces of \ke manifolds of general type. It is known that this energy functional plays a crucial role in understanding the Weil-Petersson geometry of such manifolds. When $\T_g$ is the \tei space of Riemann surfaces of genus $g\ge 2$ and $(N, h)$ a Riemannian manifold with Hermitian nonpositive curvature, it was shown by Toledo \cite{toledo12} that the energy function is plurisubharmonic. Here, we consider the higher dimensional analogy. Assume $(\XX, B, \pi)$ is the Kuranishi family of a compact \ke manifold of general type, and let $(N, h)$ be a Riemannian manifold with Hermitian nonpositive curvature. Let $E(t)$ be the energy of a harmonic map from $X_t$ to $N$ in a given homotopy class.
By using the deformation theory established in \cite{sundeform1} and the Siu-Sampson vanishing theorem  in \cite{sampson84}, we derive the first and second variation formulas of $E$ and prove its plurisubharmonicity (Theorem \ref{pshenergy}).

\bigskip
\noindent {\bf Acknowledgements.} We would like to thank T. Collins and J. Keller for their helpful comments on an earlier version of the paper.  The second named author would also like to thank S. K. Donaldson, D. H. Phong, R. Schoen and X. Wang for helpful discussions.

\section{The Kuranishi Gauge}\label{kudiv}

In this section we derive some special properties of the Kuranishi gauge on a family of compact complex manifolds when the central fiber is a Fano \ke manifold. This leads to an explicit description of the action of the automorphism group of the central fiber on the Kuranishi space via differential geometric data.

Throughout this section we assume that $\lb X, \omega_0, J_0\rb$ is a Fano manifold with complex dimension $\dim_\C X=n\geq 2$. Here $X$ is a fixed smooth manifold and we denote by $X_0=\lb X, J_0\rb$ the corresponding complex manifold. Since the canonical line bundle $K_{X_0}$ is negative, by the Serre duality and the Kodaira vanishing theorem, we have
\begin{eqnarray}\label{vanish}
H^{0,k}\lb X_0, T^{1,0}X_0\rb=0
\end{eqnarray}
for all $2\leq k\leq n$. In particular, deformations of $X_0$ are unobstructed.

By the work of Kodaira-Spencer, we know that any almost complex structure $J$ on $X$  close to $J_0$ can be described by a unique Beltrami differential $\phi\in A^{0,1}\lb X_0, T^{1,0}X_0\rb$, and $J$ is integrable if and only if
\begin{eqnarray}\label{integ}
\bar\pa_0\phi=\frac 12 \lsb \phi,\phi\rsb.
\end{eqnarray}
In order to construct a complete family of small deformations of  $X_0$, Kuranishi introduced another equation. Let $\Delta_\eps^m\subset \C^m$ be the open ball with center $0$ and radius $\eps$. For any Beltrami differential $\phi_1\in A^{0,1}\lb X_0, T^{1,0}X_0\rb$ with $\bar\pa_0\phi_1=0$, there exists $\eps>0$ such that the equation
\begin{eqnarray}\label{kueq}
\phi(t)=t\phi_1+\frac 12 \bar\pa_0^*G_0 \lsb \phi(t),\phi(t)\rsb
\end{eqnarray}
has a unique power series solution $\phi(t)=\sum_{i\geq 1}t^i\phi_i\in A^{0,1}\lb X_0, T^{1,0}X_0\rb$ which converges (in some appropriate H\"older norm) for all $t\in\Delta_\eps^1$. Here, the Green's function $G_0$ and $\bar\pa_0^*$ are operators on $X_0$ with respect to the K\"ahler metric $\omega_0$. It follows from the standard elliptic estimate and  \eqref{vanish} that each $\phi(t)$ satisfies the integrability equation \eqref{integ} and defines a complex structure on $X$. We also note that
\[
\bar\pa_0^*\lb \phi(t)-t\phi_1\rb=0.
\]

By using this construction and the Kodaira-Spencer theory, one can construct a Kuranishi family in the following way. We pick a basis $\phi_1,\cdots,\phi_m\in \h^{0,1}\lb X_0, T^{1,0}X_0\rb$, where we use $\h$ to denote the harmonic space or harmonic projection with respect to the metric $\omega_0$. Let $B=\Delta_\eps^m\subset\C^m$ be a ball with coordinates $t=\lb t_1,\cdots,t_m\rb$ and denote by
\begin{eqnarray}\label{power10}
\phi(t)=\sum_{i=1}^m t_i\phi_i+\sum_{|I|\geq 2}t^I\phi_{_I}
\end{eqnarray}
the unique solution of
\begin{eqnarray}\label{maineq}
\begin{cases}
\bar\pa_0\phi(t)=\frac 12 \lsb \phi(t),\phi(t)\rsb,\\
\bar\pa_0^*\phi(t)=0,\\
\h\lb \phi(t)\rb=\sum_{i=1}^m t_i\phi_i.
\end{cases}
\end{eqnarray}

We note that the second equation of \eqref{maineq} is the Kuranishi gauge condition, and the third equation characterizes the flat coordinate system around $0\in B$ up to affine transformations.

Now we consider the smooth manifold
\begin{eqnarray}\label{backmfd}
\XX=X\times B
\end{eqnarray}
 and define an almost complex structure $\XJ$ on $\XX$ in the following way: for each point $\lb p,t\rb\in \XX$, where $p\in X$ and $t\in B$, we let
\begin{eqnarray}\label{cxuniv}
\Omega_{(p,t)}^{1,0}\XX=\lb I+\phi(t)\rb\lb \Omega_p^{1,0}X_0\rb\oplus \pi^*\Omega_t^{1,0}B,
\end{eqnarray}
where $\phi(t)$ is given by \eqref{maineq}.
Kodaira and Spencer showed that this almost complex structure $\XJ$ on $\XX$ is integrable and $\pi:\XX\to B$ is a Kuranishi family of $X_0$. For each $t\in B$, we let $X_t=\pi^{-1}(t)$ and denote the corresponding complex structure by $J_t$.

Thanks to the works of Kuranishi \cite{ku62, ku65} and Wavrik \cite{wa69}, we have the following properties of the family $\pi:\XX\to B$; see also \cite{cata88}.

\begin{theorem}\label{versal}
Let $\pi:\XX\to B$ be the Kuranishi family of $X_0$ constructed above. Then
\begin{enumerate}
\item The Kuranishi family of $X_0$ parameterizes all small deformations of $X_0$ and is unique up to isomorphisms;

\smallskip
\item $\pi:\XX\to B$ is semiuniversal at $0\in B$;

\smallskip
\item $\pi:\XX\to B$ is complete at each point $t\in B$;

\smallskip
\item If $h^0\lb X_t, T^{1,0}X_t\rb$ is constant in $t\in B$, then the Kuranishi family is universal at each $t\in B$.
\end{enumerate}
\end{theorem}

In general, the complex structure $J_t$ is not compatible with $\omega_0$, which is viewed as a symplectic form on $X$. The compatibility property requires $\phi(t)\lrcorner\omega_0=0$. Since $\bar\pa_0^*\phi(t)=0$, a direct computation shows that $\phi(t)\lrcorner\omega_0=0$ if and only if $\dd_0\phi(t)=0$. This divergence gauge was introduced in \cite{sundeform1} and \cite{sundeform2}, where it was shown that the Kuranishi gauge $\bar\pa_0^*\phi(t)=0$ is equivalent to the divergence gauge $\dd_0\phi(t)=0$ when the fibers $X_t$ are \ke manifolds with negative or zero scalar curvature. In this section, we generalize this equivalence to the Fano case.

\begin{theorem}\label{equivmain}
Let $\lb X_0,\omega_0\rb$ be a Fano \ke manifold.
\smallskip
\begin{enumerate}

\item If $\phi(t)$ is the unique solution of equations \eqref{maineq}, then $\dd_0\phi(t)=0$ and $\phi(t)\lrcorner\omega_0=0$.

\smallskip
\item If $\phi\in A^{0,1}\lb X_0, T^{1,0}X_0\rb$ is a Beltrami differential with $\bar\pa_0\phi=\frac 12 \lsb\phi, \phi\rsb$ and $\dd_0\phi=0$, then $\bar\pa_0^*\phi=0$ and $\phi(t)\lrcorner\omega_0=0$.
\end{enumerate}
\end{theorem}

To prove this theorem, we need the following technical results.

\begin{lemma}\label{harbel}
Let $\lb X,\omega_g\rb$ be a \ka manifold.
\smallskip
\begin{enumerate}
\item[(a)] If $\phi\in A^{0,1}\lb X, T^{1,0}X\rb$ with $\bar\pa\lb \phi\lrcorner\omega\rb=0$ and $\bar\pa^*\phi=0$, then
\[
\Delta_{\bar\pa}\lb\phi\lrcorner\omega\rb=\frac{\sqrt{-1}}{2}\dd\lb\bar\pa\phi\rb+\phi\lrcorner\text{Ric}\lb\omega\rb.
\]

\item[(b)] If $\lb X,\omega_g\rb$ is Fano K\"ahler-Einstein, and $\eta\in A^{0,2}\lb X\rb$ such that $\bar\pa\eta=0$ and $\Delta_{\bar\pa}\eta=\eta$, then $\eta=0$.
\end{enumerate}
\end{lemma}

\begin{proof}
The first claim (a) follows from direct computations; we refer the reader to  \cite{sundeform1} and \cite{sundeform2} for details. To prove the second claim, by the assumptions, we have the following Bochner formula,
\[
\Delta|\eta|^2=\left |\pa\eta\right |^2+\left |\bar\nabla\eta\right |^2.
\]
This implies $\pa\eta=0$. Since $\pa^*\eta=0$, we conclude that $\Delta_\pa\eta=0$. Thus
\[
\eta=\Delta_{\bar\pa}\eta=\Delta_\pa\eta=0.
\] \end{proof}
Now we can prove Theorem \ref{equivmain}.

\begin{proof}
The proof of claim (2) is similar to that in \cite{sundeform1}. Indeed, since $\dd_0\phi=0$, we have
\begin{align*}
\begin{split}
0=\bar\pa_0\lb\dd_0\phi\rb =& \dd_0\lb\bar\pa_0\phi\rb-2\sqrt{-1}\phi\lrcorner\text{Ric}\lb\omega_0\rb \\
=& \frac 12 \dd_0\lsb\phi,\phi\rsb-2\sqrt{-1}\phi\lrcorner\omega_0\\
=& \phi\lrcorner \pa_0\lb \dd_0\phi\rb-2\sqrt{-1}\phi\lrcorner\omega_0\\
=& -2\sqrt{-1}\phi\lrcorner\omega_0.
\end{split}
\end{align*}
Together with $\dd_0\phi=0$, a direct computation shows that $\bar\pa_0^*\phi=0$.

Now we prove claim (1). Consider the power series \eqref{power10} which satisfies equations \eqref{maineq}. We will use induction on $|I|$ to show that $\dd_0\phi_{_I}=0$. If $|I|=1$, then $\phi_{_I}=\phi_i$ for some $1\leq i\leq m$ which is harmonic. Thus
\begin{eqnarray}\label{prepare10}
\bar\pa_0\lb\phi_i\lrcorner\omega_0\rb=0 \quad \text{and}\quad \bar\pa_0^*\phi_i=0.
\end{eqnarray}
Then Lemma \ref{harbel} implies that
\begin{eqnarray}\label{prepare20}
\Delta_{\bar\pa_0}\lb\phi_i\lrcorner\omega_0\rb=\frac{\sqrt{-1}}{2}\dd_0\lb\bar\pa_0\phi_i\rb+\phi_i\lrcorner\text{Ric}\lb\omega_0\rb.
\end{eqnarray}
Since $\bar\pa_0\phi_i=0$ and $\text{Ric}\lb\omega_0\rb=\omega_0$, we have
\begin{eqnarray}\label{prepare30}
\Delta_{\bar\pa_0}\lb\phi_i\lrcorner\omega_0\rb=\phi_i\lrcorner\omega_0.
\end{eqnarray}
Again by Lemma \ref{harbel}, we know that $\phi_i\lrcorner\omega_0=0$. Combining with $\bar\pa_0^*\phi_i=0$ we get $\dd_0\phi_i=0$.

Now we assume $\dd_0\phi_{_I}=0$ for all $|I|\leq k-1$. For any multi-index $I$ with $|I|=k$, we have
\begin{align*}
\begin{split}
\bar\pa_0\lb \phi_{_I}\lrcorner\omega_0\rb=& \bar\pa_0 \phi_{_I}\lrcorner\omega_0 \\
=& \frac 12\sum_{J+K=I}\lsb \phi_{_J},\phi_{_K}\rsb\lrcorner\omega_0\\
=& \frac 12\sum_{J+K=I} \lb \phi_{_J}\lrcorner\pa_0
\lb\phi_{_K}\lrcorner\omega_0\rb+\phi_{_K}\lrcorner\pa_0\lb\phi_{_J}\lrcorner\omega_0\rb\rb = 0.
\end{split}
\end{align*}
Since $\bar\pa_0^*\phi_{_I}=0$, we conclude from  Lemma \ref{harbel} that
\begin{align*}
\begin{split}
\Delta_{\bar\pa}\lb \phi_{_I}\lrcorner \omega_0\rb=& \frac{\sqrt{-1}}{2}\dd_0\lb\bar\pa_0\phi_{_I}\rb+\phi_{_I}\lrcorner \text{Ric}\lb \omega_0\rb \\
=& \frac{\sqrt{-1}}{4}\lb \sum_{J+K=I}\lsb \phi_{_J},\phi_{_K}\rsb\rb
+ \phi_{_I}\lrcorner \omega_0\\
=& \frac{\sqrt{-1}}{4}\lb \sum_{J+K=I}\phi_{_J}\lrcorner\pa_0\lb \dd_0\phi_{_K}\rb+\phi_{_K}\lrcorner\pa_0\lb \dd_0\phi_{_J}\rb\rb+ \phi_{_I}\lrcorner \omega_0\\
=& \phi_{_I}\lrcorner \omega_0,
\end{split}
\end{align*}
where we have used the fact that $\dd_0\phi_{_J}=\dd_0\phi_{_K}=0$ for all $|J|,|K|<|I|$. It then follows from Lemma \ref{harbel} that $\phi_{_I}\lrcorner\omega_0=0$. Together with the assumption $\bar\pa_0^*\phi_{_I}=0$, we conclude that $\dd_0\phi_{_I}=0$.

\end{proof}

\begin{remark}\label{generalgauge}
Let $\lb X_0,\omega_0\rb$ be a Fano manifold with $\lsb \omega_0\rsb=2\pi c_1\lb X_0\rb$ and let $\pi:\XX\to B$ be a Kuranishi family of $X_0$ defined by \eqref{backmfd} and \eqref{cxuniv} where $\phi(t)$ is the unique solution of equations \eqref{maineq}.
\begin{enumerate}

\item For any Beltrami differentials $\phi,\psi\in  A^{0,1}\lb X_0,T^{1,0}X_0\rb$ with $\phi\lrcorner\omega_0=0$ or $\psi\lrcorner\omega_0=0$, the pointwise Hermitian inner product is given by
\begin{eqnarray}\label{wp10}
\phi\cdot\bar\psi= \langle \phi,\psi\rangle_g=\phi_{\bar j}^i\bar{\psi_{\bar k}^l}g_{i\bar l}g^{k\bar j}=\phi_{\bar j}^i\bar{\psi_{\bar i}^j},
\end{eqnarray}
where $g$ is the corresponding \ka metric. 

\item If $\omega_0$ is a \ke metric, then each $J_t$ is compatible with $\omega_0$. This implies that, in the \ke case, the Kuranishi gauge is compatible with Donaldson's infinite dimensional GIT picture. In fact, let $\omega=\omega_0$ be the symplectic form on $X$ and let $\mathcal J^{int}$ be the space of integrable almost complex structures on $X$ which are compatible with $\omega$. Theorem \ref{equivmain} shows that $B$ can be viewed naturally as a slice in $\mathcal J^{int}$ containing $J_0$ via Kuranishi's construction described above.

\item Theorem \ref{equivmain} holds in more general situation if we allow appropriate twist. Let $f$ be the normalized Ricci potential satisfying
\[
\begin{cases}
\text{Ric}\lb\omega_0\rb=\omega_0+\frac{\sqrt{-1}}{2}\pa_0\bar\pa_0 f\\
\int_{X_0}f\omega_0^n=0.
\end{cases}
\]
If we define the twisted operators $\bar\pa_f^*$ and $\dd_f$ with respect to the weighted volume form $e^f\frac{\omega_0^n}{n!}$,  then the twisted Kuranishi gauge $\bar\pa_f^* \phi(t)=0$ is equivalent to the twisted divergence gauge $\dd_f \phi(t)=0$. In particular, we still have $\phi(t)\lrcorner\omega_0=0$. The proof is essentially the same as that of Theorem \ref{equivmain}.

\end{enumerate}
\end{remark}

An immediate corollary of Theorem \ref{equivmain} is the explicit expression of a Ricci potential of the \ka manifold $\lb X_t,\omega_0\rb$. This turns out to play an important role in the proof of Theorem \ref{ip10} (Theorem \ref{fkemain}). As above, let $\lb X_0,\omega_0\rb$ be a Fano \ke manifold, let $\phi(t)$ be the solution of equation \eqref{maineq} and let $\lb\XX, B,\pi\rb$ be the Kuranishi family of $\lb X_0,\omega_0\rb$ constructed above. Then Theorem \ref{equivmain} implies that the symplectic form $\omega_0$ is indeed a \ka form on $X_t$.

\begin{corollary}\label{riccipotential}
A Ricci potential of the \ka manifold $\lb X_t,\omega_0\rb$ is given by
\begin{eqnarray}\label{rcpot}
h_t=\log\det\lb I-\phi(t)\bar{\phi(t)}\rb.
\end{eqnarray}
Namely,
\begin{eqnarray}\label{rcmain}
\text{Ric}\lb X_t,\omega_0\rb=\omega_0+\frac{\sqrt{-1}}{2}\pa_t\bar\pa_t \log\det\lb I-\phi(t)\bar{\phi(t)}\rb.
\end{eqnarray}
\end{corollary}

\begin{proof}
We want to show that $-\pa_t\bar\pa_t \log \lb e^{h_t}\frac{\omega_0^n}{n!}\rb=\omega_0$. Fixing a point $t\in B$ and we let $\phi=\phi(t)$, $z=\lb z_1,\cdots,z_n\rb$ be  local holomorphic coordinates on $X_0$, $w=\lb w_1,\cdots,w_n\rb$ be local holomorphic coordinates on $X_t$, $\omega_0=\frac{\sqrt{-1}}{2}g_{i\bar j}dz_i\wedge d\bar z_j$, $g=\det\lsb g_{i\bar j}\rsb$, $A=\lsb a_{\ga i}\rsb=\lsb \frac{\pa w_\ga}{\pa z_i}\rsb$ and $B=\lsb b^{i\ga}\rsb=A^{-1}$. Then
\[
e^{h_t}\frac{\omega_0^n}{n!}=c_n \left |\det A\right |^{-2} g dw_1\wedge\cdots\wedge dw_n\wedge d\bar w_1\wedge\cdots\wedge d\bar w_n
\]
where $c_n=\lb -1\rb^{\frac{n(n-1)}{2}}\lb \frac{\sqrt{-1}}{2}\rb^n$. On the other hand, we have
\begin{eqnarray}\label{aux20}
\frac{\pa w_\ga}{\pa\bar z_j}=\phi_{\bar j}^i a_{\ga i},\ \ \ \
\frac{\pa z_i}{\pa w_\ga}=\lb I-\phi\bar\phi\rb^{ik}b^{k\ga},\ \ \ \
\frac{\pa z_i}{\pa\bar w_\gb}=-\phi_{\bar j}^i\bar{\lb I-\phi\bar\phi\rb^{jl}}\bar{b^{l\gb}},
\end{eqnarray}
where $\lb I-\phi\bar\phi\rb^{ik}$ is the $\lb i,k\rb$-entry of the matrix $\lb I-\phi\bar\phi\rb^{-1}$. By a direct computation, we have
\begin{align}\label{aux30}
\begin{split}
-\frac{\pa^2}{\pa w_\ga\pa\bar w_\gb}\log \lb c_n \left |\det A\right |^{-2} g\rb=& \frac{\pa\bar z_l}{\pa\bar w_\gb}b^{k\ga} \lsb
\frac{\pa}{\pa z_k}\phi_{\bar l}^p\mu^p+R_{k\bar l}+\frac{\pa}{\pa z_k} \lb \lb\dd_0\phi\rb_{\bar l}\rb\rsb\\
& -\frac{\pa\bar z_l}{\pa\bar w_\gb}b^{k\ga}\lb \frac{\pa}{\pa\bar z_l}-\phi_{\bar l}^i \frac{\pa}{\pa z_i}\rb\lb \mu^k\rb,
\end{split}
\end{align}
where $$\mu^k=\lb I-\phi\bar\phi\rb^{ik}\lsb \bar{\phi_{\bar i}^j}\lb \dd_0\phi\rb_{\bar j}-\lb \bar{\dd_0\phi}\rb_i\rsb,$$ and $\frac{\sqrt{-1}}{2}R_{i\bar j}dz_i\wedge d\bar z_j$ is the Ricci form of $\lb X_0,\omega_0\rb$. Since $\omega_0$ is a \ke metric on $X_0$, we have $R_{i\bar j}=g_{i\bar j}$. By Theorem \ref{equivmain}, we know $\dd_0\phi=0$ which implies $\mu^k=0$. Hence the above formula reduces to
\begin{eqnarray}\label{aux40}
-\frac{\pa^2}{\pa w_\ga\pa\bar w_\gb}\log \lb c_n \left |\det A\right |^{-2} g\rb=\frac{\pa\bar z_l}{\pa\bar w_\gb}b^{k\ga} g_{k\bar l}.
\end{eqnarray}
It remains to show that
\[
\frac{\sqrt{-1}}{2}\frac{\pa\bar z_l}{\pa\bar w_\gb}b^{k\ga} g_{k\bar l}dw_\ga\wedge d\bar w_\gb=\omega_0.
\]
Again, by Theorem \ref{equivmain}, we know $\phi_{\bar j}^ig_{i\bar l}=\phi_{\bar l}^ig_{i\bar j}$ and the above identity follows immediately from formula \eqref{aux20}.
\end{proof}

Now we look at the action of the automorphism group of $X_0$ on the Kuranishi space $B$. For the rest of this section, we assume $\omega_0$ is a \ke metric on $X_0$.

Let $G=\text{Isom}_0\lb X_0,\omega_0\rb$ be the isometry group with Lie algebra $\g$. By the work of Matsushima \cite{matsu57} and Calabi \cite{calabi2}, we know that the complexification $G^\C$ of $G$ is isomorphic to the holomorphic automorphism group $\text{Aut}_0\lb X_0\rb$ and we have $\g^{\C}\cong H^0\lb X_0, T^{1,0}X_0\rb$. Furthermore, if we let
\[
\Lambda_1^\R=\lfb f\in C^\infty\lb X_0,\R\rb\mid \lb \Delta_0+1\rb f=0\rfb
\]
be the first eigenspace of the Laplacian on $X_0$ and let $\Lambda_1^\C=\Lambda_1^\R\otimes_\R \C$, then we have
\begin{eqnarray}\label{lieagid10}
\g\cong \lfb \text{Im}\lb \nabla_0^{1,0} f\rb\big | f\in \Lambda_1^\R\rfb
\end{eqnarray}
and
\begin{eqnarray}\label{lieagid20}
\g^\C\cong \lfb  \nabla_0^{1,0} f \big | f\in \Lambda_1^\C\rfb.
\end{eqnarray}

The diffeomorphism group of $X$ acts on the space of complex structures on $X$ via pullback and thus acts locally on the set of Beltrami differentials on $X_0$ which satisfy the obstruction equation \eqref{integ}. Let $D\subset\text{Diff}_0\lb X\rb$ be a neighborhood of the identity map and let $Y=\lb X,J\rb$ be a complex manifold obtained by deforming the complex structure $J_0$ via $\phi\in A^{0,1}\lb X_0, T^{1,0}X_0\rb$. We assume $\Vert\phi\Vert$ is small and $\sigma\in D$. In \cite {ku65} Kuranishi showed that the Beltrami differential $\psi=\phi\circ\sigma$ corresponding to the complex structure $\sigma^* J$ is characterized by
\begin{equation}\label{belact}
\frac{\pa\sigma_k}{\pa\bar z_j}+\phi_{\bar l}^k\lb \sigma(z)\rb \frac{\pa\bar\sigma_l}{\pa\bar z_j}=\psi_{\bar j}^i \lb
\frac{\pa\sigma_k}{\pa z_i}+\phi_{\bar l}^k \lb \sigma(z)\rb \frac{\pa\bar\sigma_l}{\pa z_i}\rb,
\end{equation}
where $z_1,\cdots,z_n$ are local holomorphic coordinates on $X_0$. It follows that
\begin{corollary}\label{spact}
If $\sigma\in \text{Aut}_0\lb X_0\rb$ is a biholomorphism of $X_0$ then $\phi\circ\sigma=\sigma^*\phi$. If $\sigma\in \text{Aut}_0\lb Y\rb$ is a biholomorphism of $Y$ then $\phi\circ\sigma=\phi$.
\end{corollary}

Now we assume that $\sigma\in G\cap D$ is an isometry of $\lb X_0,\omega_0\rb$ and $\phi(t)$ is a solution of equation \eqref{maineq}. Then $\phi(t)\circ \sigma=\sigma^*\phi(t)$ satisfies the first two equations of \eqref{maineq} since $\sigma$ preserves $\omega_0$ and $J_0$. Thus, for each $t$ with $|t|$ small, $\sigma^*\phi(t)=\phi\lb t'\rb$ where $t'$ is characterized by $\sum_i t_i'\phi_i=\h\lb \sigma^*\phi(t)\rb$. Let $V=T_0^{1,0}B\cong H^{0,1}\lb X_0, T^{1,0}X_0\rb$. If we linearize the above action with respect to $\phi$, then we see that the linear action of $G$ on $T_0^{1,0}B$, denoted by $\rho:G\to GL(V)$, is given by
\begin{eqnarray}\label{repgp}
\rho\lb\sigma\rb\lb \lsb\phi\rsb\rb=\lsb \sigma^*\phi\rsb.
\end{eqnarray}
This is also true at the form level: $\sigma^*\phi$ is harmonic when $\sigma$ is an isometry and $\phi$ is harmonic.  The representation $\rho$ naturally extends to the representation $\rho^\C:G^\C\to GL(V)$ which is also given by \eqref{repgp}. Now we linearize the representation $\rho$ and we have the representation of Lie algebra $\rho_*:\g\to \text{End}(V)$ given by
\begin{eqnarray}\label{replag}
\rho_*\lb v\rb\lb \lsb\phi\rsb\rb=\lsb L_v\phi\rsb.
\end{eqnarray}
Again, this holds at the form level: $L_v\phi$ is harmonic when $v\in\g$ is a Killing field and $\phi$ is harmonic. This representation also extends to a representation $\rho_*^\C:\g^\C\to \text{End}\lb V\rb$.

\begin{remark}\label{actotal}
We note that, by the construction of Kuranishi family \eqref{backmfd} and \eqref{cxuniv}, both $G$ and $G^\C$ act on $\XX$ holomorphically.
\end{remark}
Note that if $v\in H^0\lb X_0, T^{1,0}X_0\rb$ is a holomorphic vector field and $\phi,\psi\in \h^{0,1}\lb X_0, T^{1,0}X_0\rb$ are harmonic Beltrami differentials, then by direct computations we have
\begin{align}\label{aux10}
\begin{split}
 L_v\phi=&\lsb v,\phi\rsb,\\
 L_{\bar v}\phi=&\bar\pa \lb \bar v\lrcorner\phi\rb,\\
 \lsb v,\phi\rsb\cdot\bar\psi=v\lb\phi\cdot\bar\psi\rb-&\dd\lb \lb v\lrcorner\bar\psi\rb\lrcorner\phi\rb+\lb v\lrcorner\bar\psi\rb\lrcorner\lb\dd\phi\rb.
\end{split}
\end{align}

Now we look at the representation $\rho_*:\g\to \text{End}\lb V\rb$. Let $\xi\in\g$ be a Killing field. By the identification \eqref{lieagid20}, there exists a unique eigenfunction $f\in \Lambda_1^\R$ such that $\xi=\text{Im}(v)$, where $$v=\nabla^{1,0}f\in H^0\lb X_0, T^{1,0}X_0\rb.$$ For any harmonic Beltrami differentials $\phi,\psi\in \h^{0,1}\lb X_0, T^{1,0}X_0\rb$, we have
\[
\langle L_\xi \phi,\psi\rangle_{L^2}=\frac{1}{2\sqrt{-1}}\lb \langle L_v \phi,\psi\rangle_{L^2}-\langle L_{\bar v} \phi,\psi\rangle_{L^2}\rb.
\]
By \eqref{aux10}, we know that
\[
\langle L_{\bar v} \phi,\psi\rangle_{L^2}=\int_{X_0} \lb \bar\pa \lb \bar v\lrcorner\phi\rb \rb\cdot\bar\psi\ dV_g=0
\]
since $\psi$ is harmonic. By integration by parts and Theorem \ref{equivmain}, we have
\begin{align*}
\begin{split}
\langle L_v \phi,\psi\rangle_{L^2}=& \int_{X_0} \lb v\lb\phi\cdot\bar\psi\rb-\dd\lb \lb v\lrcorner\bar\psi\rb\lrcorner\phi\rb+\lb v\lrcorner\bar\psi\rb\lrcorner\lb\dd\phi\rb\rb\ dV_g\\
=& \int_{X_0}  v\lb\phi\cdot\bar\psi\rb \ dV_g = \int_{X_0}\lb \dd v\rb \lb\phi\cdot\bar\psi\rb \ dV_g\\
=& - \int_{X_0}f \lb\phi\cdot\bar\psi\rb \ dV_g.
\end{split}
\end{align*}
This implies
\[
\langle L_\xi \phi,\psi\rangle_{L^2}=\frac{\sqrt{-1}}{2}\int_{X_0}f \lb\phi\cdot\bar\psi\rb \ dV_g.
\]
Let
\begin{eqnarray}\label{spaceq}
Q=\lfb \phi\cdot\bar\psi\mid \phi,\psi\in \h^{0,1}\lb X_0, T^{1,0}X_0\rb\rfb\subset C^\infty\lb X_0\rb.
\end{eqnarray}
Since $L_\xi \phi$ is harmonic, we know that $L_\xi \phi=0$ if and only if $f\bot_{L^2} Q$.

In conclusion, we have proved the following
\begin{corollary}\label{trivm}
The representation $\rho_*$ is trivial (and thus $\rho_*^\C$, $\rho$ and $\rho^\C$ are trivial) if and only if $\Lambda_1^\R\bot_{L^2} Q$ (and thus $\Lambda_1^\C\bot_{L^2} Q$).
\end{corollary}

\section{Small Deformation of Fano \ke Manifolds}\label{defoke}

Throughout this section, we assume $\lb X_0, \omega_0\rb$ is a Fano \ke manifold and denote by $\lb \XX, B,\pi\rb$ the Kuranishi family with respect to $\omega_0$ as constructed in Section \ref{kudiv}. An important question concerning the geometry of the moduli space of $X_0$ is the existence of \ke metrics on small deformations of $\lb X_0, \omega_0\rb$. By using the implicit function theorem, Koiso \cite{koiso1} showed that any small deformation of $X_0$ admits a \ke metric, provided the automorphism group of $X_0$ is discrete. The case that $X_0$ has non-trivial holomorphic vector fields is much more delicate. In \cite{szeke2010} Szekelyhidi showed that a small deformation of a cscK manifold admits a cscK metric if and only if it is $K$-polystable. A similar result was established by Br\"onnle \cite{bro} in terms of the polystability of the action of the automorphism group on the Kuranishi space.  Later, it was proved by Chen, Donaldson and Sun \cite{cds1, cds2, cds3} that the existence of a \ke metric on a Fano manifold $X$ is equivalent to the $K$-stability of $X$.
However, it is highly nontrivial to check the $K$-stability of a Fano manifold in general.

In this section, we provide new and simple necessary and sufficient conditions on the existence of \ke metrics on small deformations of $X_0$ as stated in Theorem 1.1.

\begin{theorem}\label{fkemain}
Let $\lb X_0,\omega_0\rb$ be a Fano \ke manifold and let $\lb \XX, B,\pi\rb$ be the Kuranishi family with respect to $\omega_0$. By shrinking $B$ if necessary, the following statements are equivalent:

\smallskip
\begin{enumerate}
\item $X_t$ admits a \ke metric for each $t\in B$;
\smallskip

\item The dimension $h^0\lb X_t, T^{1,0}X_t\rb$ of the space of holomorphic vector fields on $X_t$ is independent of $t$ for all $t\in B$;

\smallskip

\item The automorphism group $\text{Aut}_0\lb X_t\rb$ is isomorphic to $\text{Aut}_0\lb X_0\rb$ for all $t\in B$.
\end{enumerate}
\end{theorem}

\begin{proof}
Firstly, we assume that $X_t$ admits \ke metrics for each $t\in B$. By Remark \ref{generalgauge}, we know that $\omega_0$ defines a \ka metric on $X_t$. We shall show that
\begin{eqnarray}\label{key100}
\text{Isom}_0\lb X_0,\omega_0\rb=\text{Isom}_0\lb X_t, \omega_0\rb
\end{eqnarray}
for each $t\in B$. Once we have this, then statements (2) and (3) follow from the upper semi-continuity of $h^0\lb X_t, T^{1,0}X_t\rb$ as a function of $t$ \cite{ks57}. Indeed, after shrinking $B$ we can assume that $h^0\lb X_t,T^{1,0}X_t\rb\leq h^0\lb X_0,T^{1,0}X_0\rb$ for all $t\in B$. We know that $\lb \text{Isom}_0\lb X_t, \omega_0\rb\rb^\C$ is a subgroup of $\text{Aut}_0\lb X_t\rb$ for each $t$ and $\lb\text{Isom}_0\lb X_0,\omega_0\rb\rb^\C\cong \text{Aut}_0\lb X_0\rb$. Since
\begin{align*}
\begin{split}
\dim_\R \text{Isom}_0\lb X_t,\omega_0\rb\leq &\dim_\C \text{Aut}_0\lb X_t\rb =h^0\lb X_t,T^{1,0}X_t\rb \\
\leq & h^0\lb X_0,T^{1,0}X_0\rb =\dim_\C \text{Aut}_0\lb X_0\rb \\
=& \dim_\R \text{Isom}_0\lb X_0,\omega_0\rb,
\end{split}
\end{align*}
identity \eqref{key100} would imply that $h^0\lb X_t,T^{1,0}X_t\rb= h^0\lb X_0,T^{1,0}X_0\rb$ and $\lb \text{Isom}_0\lb X_t, \omega_0\rb\rb^\C\cong\text{Aut}_0\lb X_t\rb$ for all $t\in B$. It then follows that $\text{Aut}_0\lb X_t\rb\cong \text{Aut}_0\lb X_0\rb$ for all $t\in B$ since they are complexifications of the same compact Lie group.

To prove \eqref{key100}, it suffices to show that each isometry $\sigma\in G= \text{Isom}_0\lb X_0,\omega_0\rb$, viewed as a diffeomorphism of $X$, is also an isometry of $\lb X_t,\omega_0\rb$. This will give us a natural embedding
\begin{eqnarray}\label{key200}
G\hookrightarrow \text{Isom}_0\lb X_t,\omega_0\rb
\end{eqnarray}
and  \eqref{key100} follows from the dimensional reason as above.

By the discussion in Section \ref{kudiv}, since each isometry $\sigma\in G$ preserves both the \ke metric $\omega_0$ and the complex structure $J_0$, it preserves all operators which are canonically associated to $\omega_0$ and $J_0$. Hence, such $\sigma$ maps each solution of equation \eqref{maineq} to another solution since these solutions are given by the Kuranishi equation \eqref{kueq}. On the other hand, since a solution $\phi(t)$ of equation \eqref{kueq} is determined by its harmonic part $\mathbb H\lb \phi(t)\rb=\sum_{i=1}^m t_i\phi_i$, it is enough to show that the action of $G$ (or $G^\C$) on $B$, as described in Section 2, is trivial. Indeed, this would imply  that, for each $t\in B$ and $\sigma\in G$, we have $\sigma^*\lb \sum_{i=1}^m t_i\phi_i\rb= \sum_{i=1}^m t_i\phi_i$ and thus, by the uniqueness of solution of the Kuranishi equation \eqref{kueq}, we have $\sigma^*\phi(t)=\phi(t)$ which implies $\sigma$ preserves the complex structure $J_t$. Since $\sigma$ preserves $\omega_0$, we conclude that $\sigma\in \text{Isom}_0\lb X_t,\omega_0\rb$.

It remains to show that the action of $G^\C$ on $B$ is trivial. Let us denote by
\[
V=T_0^{1,0}B\cong H^{0,1}\lb X_0, T^{1,0}X_0\rb
\]
as before. If the action of $G^\C$ on $B$ is nontrivial, then there exists a subgroup $\lambda:\C^*\to G^\C$ whose action on $V$ is nontrivial. We can then pick a basis $e_1,\cdots,e_m$ of $V$ such that
$$\lambda(s)\lb e_i\rb=s^{\kappa_i}e_i, \quad s\in \C^*$$  with $\kappa_i\in\Z$ for each $i$. It follows that at least one of the $\kappa_i$'s is nonzero. Replacing $\lambda$ by $\lambda^{-1}$ if necessary, we can assume $\kappa_i>0$ for some $i$. Let
\[
\Delta_\eps=\lfb \lb 0,\cdots,0,t_i,0,\cdots,0 \rb\big | \left | t_i\right |<\eps\rfb \subset B
\]
and pick some $t'\in \Delta_\eps^*=\Delta_\eps\setminus \lfb 0\rfb$.

Let $\XX'=\XX\mid_{\Delta_\eps}$ and consider the subfamily $\lb \XX', \Delta_\eps,\pi\rb$ with an action of $H=\lfb s\in C^*\mid |s|<1\rfb$ on $\Delta_\eps$ given by $\lambda(s)(t)=s^{\kappa_i}t$.
We note that $X_t$ is biholomorphic to $X_{t'}$ if $t\ne 0$ because of the action of $H$. Furthermore, $X_0$ is not biholomorphic to $X_{t'}$. To see this, we note that, by Theorem \ref{versal}, the Kodaira-Spencer map $KS_t: T_t^{1,0}B\to H^{0,1}\lb X_t, T^{1,0}X_t\rb$ is an isomorphism if $t=0$, and is surjective if $t\ne 0$. The above argument shows that the deformation of $X_{t'}$ is trivial along at least one direction due to the action of $\C^*$. Thus
\[
h^{0,1}\lb X_{t'}, T^{1,0}X_{t'}\rb< h^{0,1}\lb X_{0}, T^{1,0}X_{0}\rb.
\]
This shows that $X_0$ is not biholomorphic to $X_{t'}$. By Remark \ref{actotal}, we also get the action of $H$ on $\XX'$. The family $\lb \XX',\Delta_\eps,\pi\rb$ naturally extends to a family $\lb \XX'',\C,\pi\rb$ with a $\C^*$ action on the base $\C$ with weight $\kappa_i$ and the corresponding action on $\XX''$. By the standard argument of base change, we can assume $\kappa_i=1$ and we get a nontrivial test configuration of $\lb X_{t'}, K_{ X_{t'}}^{-k}\rb$, where the $\C^*$ action on the line bundle $K_{\XX''/\C}^{-k}$ is the induced one. Since $X_{t'}$ admits a \ke metric, it is $K$-polystable (\cite{tian97, ber2016}). Now the central fiber of the nontrivial test configuration $\lb \XX'',\C,\pi\rb$ is $X_0$, which also admits a \ke metric and thus the Futaki invariant is $0$. This is a contradiction, thus statements (2) and (3) hold.

Conversely, it is obvious that statement (3) implies (2), so it remains to show that (2) implies (1), namely if the dimension $h^0\lb X_t, T^{1,0}X_t\rb$ of the space of holomorphic vector fields on $X_t$ is independent of $t$:
\[h^0\lb X_t, T^{1,0}X_t\rb=h^0\lb X_0, T^{1,0}X_0\rb=l \quad \text{for all} \ t\in B, \]
then each $X_t$ admits a \ke metric. Pick a basis $\lfb v_1,\cdots,v_l\rfb$ of $H^0 \lb X_0, T^{1,0}X_0\rb$. By the above assumption and the work of Kodaira \cite{kobook}, we can extend each $v_i$ to $v_i(t)\in A^0\lb X_0, T^{\C}X_0\rb$ such that $v_i(t)\in H^0\lb X_t, T^{1,0}X_t\rb$ and $v_i(t)$ depends on $t$ holomorphically. By continuity, and by shrinking $B$ if necessary, we know that $\lfb v_1(t),\cdots,v_l(t)\rfb$ span $H^0\lb X_t, T^{1,0}X_t\rb$ for each $t\in B$.

Now we define a map
\[
\tau_t:A^0 \lb X_0, T^{1,0}X_0\rb\to A^0 \lb X_t, T^{1,0}X_t\rb
\]
by
\[
\tau_t(v)=\lb I-\phi(t)\bar{\phi(t)}\rb^{-1}(v)-\bar{\phi(t)}\lb \lb I-\phi(t)\bar{\phi(t)}\rb^{-1}(v)\rb.
\]
Then $\tau_t$ is a linear isomorphism for each $t\in B$. Let $\wt v_i(t)=\tau_t^{-1}\lb v_i(t)\rb$. Since $\bar\pa_t v_i(t)=0$, a direct computation shows that
\begin{eqnarray}\label{30}
\bar\pa_0\wt v_i(t)=-\lsb \wt v_i(t),\phi(t)\rsb.
\end{eqnarray}
Since $\phi(0)=0$, we have
\begin{align*}
\begin{split}
\bar\pa_0\lb \frac{\pa}{\pa t_k}\bigg |_{t=0}\wt v_i(t)\rb = & \frac{\pa}{\pa t_k}\bigg |_{t=0}\bar\pa_0\wt v_i(t)
= -\frac{\pa}{\pa t_k}\bigg |_{t=0}\lsb \wt v_i(t),\phi(t)\rsb
= -\lsb v_i,\phi_k\rsb.
\end{split}
\end{align*}
This implies that the cohomology class $\lsb \lsb v_i,\phi_k\rsb\rsb=0$ for all $1\leq i\leq l$ and $1\leq k\leq m$. Thus, by (\ref{aux10}), the action of the Lie algebra $\g$ on $H^{0,1}\lb X_0, T^{1,0}X_0\rb$ given by (\ref{replag}) is trivial which implies that the $G$-action (\ref{repgp}) on $T_0^{1,0}B\cong H^{0,1}\lb X_0, T^{1,0}X_0\rb$ is trivial. By the previous arguments, we have the identification $\text{Isom}_0\lb X_0,\omega_0\rb=\text{Isom}_0\lb X_t,\omega_0\rb$ for each $t\in B$.

We can now restrict our attention to $G$-invariant \ka potentials and apply the implicit function theorem as in \cite{lesi1994, rst1, szeke2010, bro} (which can be further traced back, e.g., to the work of Donaldson-Kronheimer \cite{dk1990}). More specifically, by the work in \cite{lesi1994} (see also Corollary 1  in  \cite{rst1}), the above identification  $\text{Isom}_0\lb X_0,\omega_0\rb=\text{Isom}_0\lb X_t,\omega_0\rb$ leads to the existence of an extremal metric on each $X_t$. On the other hand,  by Corollary \ref{riccipotential}, we know that $h_t=\log\det\lb I-\phi(t)\bar{\phi(t)}\rb$ is a Ricci potential of $\lb X_t,\omega_0\rb$. It follows from Corollary \ref{spact} that each $\sigma\in\text{Aut}_0\lb X_t\rb$ preserves $\phi(t)$, hence the Ricci potential $h_t$ is a $\sigma$-invariant function. Thus, for each $\xi\in H^0\lb X_t, T^{1,0}X_t\rb$, we have $\xi\lb h_t\rb=0$ and the vanishing Futaki invariant \cite{futaki}:
\[
f_{X_t}\lb \omega_0,\xi\rb=\int_{X_t}\xi\lb h_t\rb\frac{\omega_0^n}{n!}=0.
\]
Therefore,  the extremal metric on $X_{t}$ must be a K\"ahler-Einstein metric. This proves statement (1) and concludes the proof of Theorem \ref{fkemain}.
\end{proof}

\begin{remark}\label{multiway}
As discussed in \cite{csyz2019}, under any of the equivalent conditions in Theorem \ref{fkemain},  any \ke metric $\omega_0$ on $X_0$ can be extended to a smooth family $\lfb\omega_t\rfb_{t\in B}$ such that $\omega_t$ is a \ke metric on $X_t$ for each $t\in B$.
\end{remark}

\begin{remark}
Szekelyhidi \cite{szeke2010} showed that if $X'$ is a sufficiently small deformation of a Fano K\"ahler-Einstein manifold $X$,  then either $X'$ admits a K\"ahler-Einstein metric or there is a test configuration for $X'$ with smooth central fibre $X''$.  Moreover,  $X''$ admits a K\"ahler-Einstein metric and it is itself a small deformation of $X$. Combining this result of Szekelyhidi and the assumption that $h^0\lb X_t, T^{1,0}X_t\rb$ is independent of $t$, one can give an alternative proof of $``$(2) $\Longrightarrow$ (1)$"$ in Theorem \ref{fkemain}.
\end{remark}

Theorem \ref{fkemain} and Theorem \ref{versal} immediately imply the following universal property of the Kuranishi family.
\begin{corollary}\label{gpnojp}
Let $\lb X_0,\omega_0\rb$ be a Fano \ke manifold and let $\lb \XX, B,\pi\rb$ be the Kuranishi family with respect to $\omega_0$. If $X_t$ admits a \ke metric for each $t\in B$, then the family $\lb \XX, B,\pi\rb$ is universal at each $t$.
\end{corollary}


\section{Curvature of the $L^2$ Metrics on Direct Image Sheaves}\label{wpmse}

The \wpm metric is a $L^2$ metric on the parameter space of a family of complex manifolds which admit certain canonical metrics. It was first introduced by Weil to study the moduli spaces of hyperbolic Riemann surfaces based on the Petersson pairing. See, e.g.,  \cite{csyz2019} for a brief survey on certain aspects of the \wpm metric.

In general, we consider a complex analytic family $\lb \mathcal Y, D,p\rb$ of compact complex manifolds, where $D\subset \C^m$ is the parameter space, and we let $Y_s=p^{-1}(s)$ for each point $s\in D$. If we assume that each fiber $Y_s$ admits a \ke metric $\omega_s$, then we can define the \wpm metric  in the following way. For any $s\in D$ and $u,v\in T_s^{1,0}D$, we let $\phi,\psi\in \h^{0,1}\lb Y_s, T^{1,0}Y_s\rb$ be the harmonic representatives of the Kodaira-Spencer classes $KS_s(u)$ and $KS_s(v)$ respectively, where we use the chosen \ke metric $\omega_s$ on $Y_s$ to determine $\phi$ and $\psi$. Then the \wpm metric  $\omega_{_{WP}}$ is given by
\begin{eqnarray}\label{wpdefn}
h_s\lb u,v\rb=\int_{Y_s}\langle \phi,\psi\rangle_{\omega_s}\frac{\omega_s^n}{n!}.
\end{eqnarray}

When $Y_s$ is a \ke manifold of general type or a polarized Calabi-Yau manifold, there is a unique \ke metric on $Y_s$. Therefore, in this case, $\phi$ and $\psi$ are uniquely determined and the \wpm metric is well-defined. Furthermore, for any submanifold $D'\subset D$, when we consider the restricted family $\lb \mathcal Y\mid_{D'}, D',p\rb$, the \wpm metric on $D'$ defined by \eqref{wpdefn} is just the restriction of the \wpm metric on $D$ to $D'$. In fact, one can define the canonical $L^2$ metric on $H^{0,1}\lb Y_s, T^{1,0}Y_s\rb$ in this case by using the unique \ke metric on $Y_s$, even when there are obstructions on deforming the complex structure on $Y_s$. This generalization of the classical \wpm metric plays an important role in studying the moduli space of $Y_s$.


Let $\lb X_0,\omega_0\rb$ be a Fano \ke manifold and let $\lb \XX, B,\pi\rb$ be the Kuranishi family constructed in Section \ref{kudiv}. We assume that each $X_t$ admits a \ke metric. Then, by Theorem \ref{fkemain} and Remark \ref{multiway}, we know that each \ke metric on $X_0$ can be extended to a smooth family of \ke metrics. In this case it is not hard to show that the \wpm metric is well-defined, namely it is independent of the choice of \ke metrics on each $X_t$. 
In fact, following the classical approach, if $\{\omega(s)\}$ is  any family of \ke metrics on $X_0$ and let $\phi_s$ and $\psi_s$ be harmonic representatives of any two given Kodaira-Spencer classes with respect to $\omega(s)$. Then a simple computation show that
\[\frac{d}{ds}\bigg |_{s=0}\langle \phi_s,\psi_s\rangle_{L^2\lb \omega(s)\rb}=0\] for all choices of $\lfb\omega(s)\rfb$ if and only if $\lfb \phi_i\bar\phi_j\rfb \bot_{L^2\lb\omega(0)\rb}\Lambda_1^\R\lb\omega(0)\rb$ and the latter condition  is guaranteed by Theorem \ref{fkemain} and Corollary \ref{trivm}.

Now we turn our attention to the approximation of the \wpm metric.
Given a smooth family $\lb\XX, B,\pi\rb$ of \ke manifolds of general type, it was shown in \cite{sundeform1} that the Ricci curvatures of the $L^2$ metrics, induced by the fiberwise \ke metrics on the direct image bundle $R^0\pi_* K_{\XX/B}^k$, converge to the \wpm metric after an appropriate normalization. There are two steps involved in establishing the curvature formula of the $L^2$ metrics. The first step is to extend sections in $H^0\lb X_0, K_{X_0}^k\rb$ to $H^0\lb X_t, K_{X_t}^k\rb$ in a canonical way in order to obtain local holomorphic sections of $R^0\pi_* K_{\XX/B}^k$. We note that the background smooth pair of $\lb X_t, K_{X_t}\rb$ is independent of $t$.
In \cite{sundeform1},  a slight different notion inspired by the work in Todorov \cite{tod1} was used.
This technique can be directly applied to more general situations, e.g.,  a family $\lb\mathcal Y, D,p\rb$ of compact complex manifolds and a relative ample holomorphic line bundle $\mathcal L$ over $\mathcal Y$. In \cite{sundeform2}, this idea was used to construct local holomorphic sections of the bundle over $D$ whose fiber at $s\in D$ is $H^0\lb Y_s, \lb \mathcal L^k\mid_{Y_s}\rb\rb$. The second step is to find deformation of the \ke metrics with respect to the Kuranishi-divergence gauge.

It turns out that similar results hold in our situation if we replace the relative canonical bundle used in \cite{sundeform1} by the relative anti-canonical bundle.
Let $\lb X_0,\omega_0\rb$ be a Fano \ke manifold and let $\lb\XX,B,\pi\rb$ be the Kuranishi family with respect to $\omega_0$.  Note that for any positive integer $k$, by the Serre duality and Kodaira vanishing theorems, we have
\begin{eqnarray}\label{vanish10}
h^i\lb X_t, K_{X_t}^{-k}\rb=h^{n-i}\lb X_t, K_{X_t}^{k+1}\rb=0
\end{eqnarray}
for all $1\leq i\leq n$ since $K_{X_t}^{k+1}$ is negative. Thus, by the Riemann-Roch theorem, we know that $h^0\lb X_t, K_{X_t}^{-k}\rb$ remain constant for all $t\in B$. This implies that the direct image sheaf $R^0\pi_* K_{\XX/B}^{-k}$ is a holomorphic vector bundle over $B$ . We denote this bundle by $E_k$ and its rank by $N_k$.

Similar to the work in \cite{tod1}, we define the linear map $\sigma_t:A^0\lb X_0, K_{X_0}^{-k}\rb\to A^0\lb X_t, K_{X_t}^{-k}\rb$ by
\begin{eqnarray}\label{ext10}
\sigma_t(s)=\lb \det\lb I-\phi(t)\bar{\phi(t)}\rb\rb^{-k}\lb s^{\frac 1k}\lrcorner e^{-\bar{\phi(t)}}\rb^k.
\end{eqnarray}
It is easy to see that $\sigma_t$ is well-defined and is an isomorphism if $|t|$ is small. Furthermore, a direct computation shows that $\sigma_t(s)$ is a holomorphic section of $K_{X_t}^{-k}$ if and only if
\begin{eqnarray}\label{obs10}
\bar\pa_0 s=\phi(t)\lrcorner\nabla_0 s,
\end{eqnarray}
where $\nabla_0$ is the metric connection on $K_{X_0}^{-k}$ induced by the \ke metric on $X_0$. By using \eqref{vanish10}, equation \eqref{obs10} can be solved inductively. Indeed, given any holomorphic section $s\in H^0\lb X_0, K_{X_0}^{-k}\rb$, we look for a power series solution
\begin{eqnarray}\label{hol10}
s(t)=s+\sum_{|I|\geq 1}s_{_I} t^I\in A^0\lb X_0, K_{X_0}^{-k}\rb
\end{eqnarray}
to equation \eqref{obs10} with normalization $\h_0\lb s(t)\rb=s$. By induction, it is not hard to see that
\begin{eqnarray}\label{hol20}
s_{_I}=\bar\pa_0^* G_0\lb \sum_{J+K=I}\phi_{_J}\lrcorner \nabla_0 s_{_K}\rb.
\end{eqnarray}
Furthermore, standard elliptic estimates imply that the power series \eqref{hol10} converges in any $C^{p,\ga}$ norm when $t$ is sufficiently small. Similar to the work in \cite{sundeform1}, we have

\begin{theorem}\label{holext}
For any holomorphic section $s\in H^0\lb X_0, K_{X_0}^{-k}\rb$, the power series solution \eqref{hol10} satisfies $\h_0\lb s(t)\rb=s$ and $\sigma_t\lb s(t)\rb\in H^0\lb X_t, K_{X_t}^{-k}\rb$ for each $t\in B$.
Furthermore, by shrinking $B$ if necessary, if $\lfb s_i\rfb_{1\leq i\leq N_k}\subset H^0\lb X_0, K_{X_0}^{-k}\rb$ is a basis then $\lfb \sigma_t\lb s_i(t)\rb\rfb_{1\leq i\leq N_k}\subset H^0\lb X_t, K_{X_t}^{-k}\rb$ is also a basis for all $t\in B$.
\end{theorem}

\begin{remark}
A direct computation shows that
\begin{align*}
\begin{split}
\phi_{_J}\lrcorner\nabla_0 s_{_K}= & \dd_0\lb \phi_{_J}\otimes s_{_K}\rb -\lb\dd_0 \phi_{_J}\rb\otimes s_{_K}\\
 =& \dd_0\lb \phi_{_J}\otimes s_{_K}\rb,
\end{split}
\end{align*}
where the last equality follows from Theorem \ref{equivmain}. Thus formula \eqref{hol20} is equivalent to
\begin{eqnarray}\label{hol50}
s_{_I}=\bar\pa_0^* G_0\lb \sum_{J+K=I}\dd_0\lb \phi_{_J}\otimes s_{_K}\rb\rb.
\end{eqnarray}
\end{remark}

If we assume each $X_t$ admits a \ke metric $g_t$ with volume form $V_t$, then the $L^2$ metric $H_k(V)$ on
$E_k=R^0\pi_* K_{\XX/B}^{-k}$ is given by
\begin{eqnarray}
\langle s_1,s_2\rangle_{H_k(V)}=\int_{X_t} \langle s_1,s_2\rangle_{g_t^k}dV_t
\end{eqnarray}
for each $t\in B$ and $s_1, s_2\in H^0\lb X_t, K_{X_t}^{-k}\rb$, where $V=\lfb V_t\rfb_{t\in B}$ and $g_t^k$ is the metric on $K_{X_t}^{-k}$ induced by the \ke metric $g_t$ on $X_t$. It is clear that the $L^2$ metric $H_k(V)$ on $E_k$ depends on the choice of the smooth family of fiberwise \ke metrics.

In order to compute the curvature of the $L^2$ metric, we need the deformation formulas of $V_t$. By using the Kuranishi-divergence gauge,  we view each $X_t$ as the background smooth manifold $X$ equipped with the complex structure $J_t$ obtained by deforming the complex structure on $X_0$ via $\phi(t)$. Thus we can view $\lfb V_t\rfb_{t\in B}$ as families of differential forms on $X$. Similar to the work in \cite{sundeform1}, by deforming the corresponding Monge-Amp\'ere equation, we have

\begin{theorem}\label{kexp}
Let $\lb X_0,\omega_0\rb$ be a Fano \ke manifold with Kuranishi family $\lb \XX,B,\pi\rb$. We assume that each $X_t$ admits a \ke metric. Let $V=\lfb V_t\rfb$ be a smooth family of \ke volume forms and we write $V_t=e^\rho\det\lb I-\phi(t)\bar{\phi(t)}\rb V_0$ for some $\rho\in C^\infty \lb X_0,\R\rb$. Then $\rho$ has an expansion of the form
\begin{eqnarray}\label{33}
\rho=\sum_i t_i\rho_i+\sum_j\bar t_j\bar\rho_j+\sum_{i,j}t_i\bar t_j\rho_{i\bar j}+O\lb t_it_k\rb+O\lb \bar t_j\bar t_l\rb +O\lb |t|^3\rb,
\end{eqnarray}
where
\begin{enumerate}
\item $\lb \Delta_0+1\rb\rho_i=0$;
\item $\lb \Delta_0+1\rb\rho_{i\bar j}=\phi_i\cdot\bar\phi_j-g_0^{\ga\bar\gb}g_0^{\gm\bar\gd}\pa_\ga\pa_{\bar\gd}\rho_i\pa_\gm\pa_{\bar\gb}\bar\rho_j$.
\end{enumerate}
\end{theorem}

\begin{remark}
Since the \ka forms $\{\omega_t$\} of the \ke metrics $\{g_t\}$ are given by $\omega_t=-\pa_t\bar\pa_t \log V_t$, formula \eqref{33} also leads to the expansion of the \ke \ka forms $\{\omega_t\}$. Furthermore, we can also eliminate both the $t_i t_k$ and $\bar t_j\bar t_l$ terms in the expansion \eqref{33} by modifying $V$ via a biholomorphism of $\XX$. However, we do not need these facts in the following discussion.
\end{remark}

To effectively compute the curvature of the $L^2$ metric $\langle \cdot,\cdot\rangle_{H_k\lb V\rb}$ for any chosen family of \ke volume forms $V$, we need to adjust the total space $\XX$ without altering the Kuranishi gauge. We shall consider a certain special  biholomorphism $F:\XX\to\XX$ which covers the identity map of $B$ and we let $F_t=F\mid_{X_t}\in \text{Aut}\lb X_t\rb$. By Corollary \ref{spact},  we know that $F_t$ preserves $\phi(t)$. Such a map $F$ would induce a biholomorphic bundle map $\wt F:E_k\to E_k$, which is indeed a Hermitian isometry $\wt F: \lb E_k, \langle \cdot,\cdot\rangle_{H_k\lb V\rb}\rb\to \lb E_k, \langle \cdot,\cdot\rangle_{H_k\lb \lb F^{-1}\rb^* V\rb}\rb$.

For any smooth family $V$ of \ke volume forms, we let $\rho=\rho^V$ be the function as in Theorem \ref{kexp}. The family $V$ is said to be normalized if $\rho_i^V=0$ for each $i$. We now construct the special biholomorphism $F$ of $\XX$, covering the identity map of $B$, such that $F^* V$ is normalized. For a given family $V$, by Theorem \ref{kexp} we know that $\rho_i^V \in\Lambda_1^\C$ is an eigenfunction of $\Delta_0+1$, hence $\mu_i=\nabla_0^{1,0}\rho_i^V\in H^0\lb X_0, T^{1,0}X_0\rb$ is a holomorphic vector field on $X_0$. Since we have assumed that each $X_t$ admits a \ke metric, by Theorem \ref{fkemain} we know that $h^0\lb X_0, T^{1,0}X_0\rb=h^0 \lb X_t, T^{1,0}X_t\rb$ for each $t\in B$. It follows from Kodaira's stability theorem that each $\mu_i$ can be extended to a family $\mu_i(t)$ of vector fields such that
\begin{enumerate}
\item[(i)] $\mu_i(t)\in H^0\lb X_t, T^{1,0}X_t\rb$ for each $t\in B$, and

\smallskip
\item [(ii)] $\mu_i(t)$ depends on $t$ holomorphically.
\end{enumerate}
We let $\mu(t)=\sum_i t_i\mu_i(t)\in H^0\lb \XX, T_{\XX/B}^{1,0}\rb$ and let $F$ be the time-one flow of $\mu(t)$. Since $\frac{\pa}{\pa t_i}\bigg |_{t=0}F=\mu_i$ and $\dd_0\mu_i=-\Delta_0\rho_i^V=\rho_i^V$, it follows from direct computations that $F^*V$ is normalized. Thus, to compute the curvature of $\lb E_k, \langle \cdot,\cdot\rangle_{H_k\lb V\rb}\rb$, we can always assume that $V$ is normalized. In this case, it follows from Theorem \ref{kexp} and $\rho_i^V=0$ that $\lb \Delta_0+1\rb\rho_{i\bar j}^V=\phi_i\cdot\bar\phi_j$. We denote by $\lb \Delta_0+1\rb^{-1}\lb \phi_i\cdot\bar\phi_j\rb$ the unique solution of this equation which is perpendicular to $\Lambda_1^\C$. It then follows that
\begin{eqnarray}\label{secord}
\rho_{i\bar j}^V=\lb \Delta_0+1\rb^{-1}\lb \phi_i\cdot\bar\phi_j\rb+\nu_{i\bar j}^V
\end{eqnarray}
for some $\nu_{i\bar j}^V\in \Lambda_1^\C$.

The above discussion leads to the approximation of the Weil-Petersson metric $\omega_{_{WP}}$ on the parameter sapce $B$ by the Ricci curvatures of the $L^2$ metrics. Such approximations can be seen via the Knudsen-Mumford expansion \cite{km1976, zhang96, prs08} and the work of Schumacher \cite{schu1993}. Here we give a simple and direct proof. Moreover, our method gives the curvature tensor of the $L^2$ metrics on the direct image sheaves rather than their determinant bundles. 

In the following, we will use $\Box_0$ to denote the Hodge Laplacian on bundles over $X_0$ with respect to metrics induced by the \ke metric $\omega_0$.

\begin{theorem}\label{quanmain}
Let $\lb X_0,\omega_0\rb$ be a Fano \ke manifold with Kuranishi family $\lb \XX,B,\pi\rb$. We assume that each $X_t$ admits a \ke metric. Let $\lfb s_{_\ga}\rfb\subset H^0\lb X_0, K_{X_0}^{-k}\rb$ be a basis, $V$ be a smooth family of \ke volume forms, and $\rc_k=\rc \lb E_k, H_k\lb V\rb\rb$. Then
\begin{eqnarray}\label{quanmain20}
\lim_{k\to\infty}\frac{\pi^n}{k^{n+1}}\rc_k=-\omega_{_{WP}}.
\end{eqnarray}
\end{theorem}

\begin{proof}

By the above discussion,  we can assume that $V$ is normalized. Let $\nu_{i\bar j}=\nu_{i\bar j}^V$ be the function given by equation \eqref{secord}. We first show that the curvature tensor of the $L^2$ metric $H_k(V)$ on $E_k$ is given by
\begin{align}\label{quanmain10}
\begin{split}
R_{\ga\bar\gb i\bar j}(0)= & \lb k+1\rb \int_{X_0}\langle \lb \Box_0+k+1\rb^{-1}\lb \phi_i\otimes s_{_\ga}\rb, \phi_j\otimes s_{_\gb}\rangle_{g_0} dV_0\\
-& \lb k+1\rb \int_{X_0}\langle s_{_\ga}, s_{_\gb}\rangle_{g_0^k}\lb \lb \Delta_0+1\rb^{-1}\lb \phi_i\cdot\bar\phi_j\rb+\nu_{i\bar j}\rb dV_0.
\end{split}
\end{align}
To prove this formula, since the curvature of $H_k$ is tensorial, we can use the local sections of $E_k$ constructed in Theorem \ref{holext} to compute it. For each $s_{_\ga}$, let $s_{_\ga}(t)\subset A^0\lb X_0, K_{X_0}^{-k}\rb$ be the sections constructed by formulas \eqref{hol10} and \eqref{hol20} and let $h_{\ga\bar\gb}(t)=\langle \sigma_t \lb s_{_\ga}(t)\rb, \sigma_t \lb s_{_\gb}(t)\rb\rangle_{H_k\lb V\rb}$. By Theorems \ref{holext} and \ref{kexp}, we have
\[
h_{\ga\bar\gb}(t)=\int_{X_0}\langle s_{_\ga}(t), s_{_\gb}(t)\rangle_{g_0^k}e^{\lb k+1\rb\rho}\det\lb I-\phi(t)\bar{\phi(t)}\rb dV_0,
\]
where $\rho$ is the function defined by
\[
V_t=e^\rho \det\lb I-\phi(t)\bar{\phi(t)}\rb V_0.
\]
Since $V$ is normalized, by formula \eqref{hol50}, we have
\begin{align}\label{firstord}
\begin{split}
\frac{\pa h_{\ga\bar\gb}}{\pa t_i}\bigg |_{t=0}=& \int_{X_0}\langle \bar\pa_0^* G_0 \dd_0\lb \phi_i\otimes s_{_\ga}\rb, s_{_\gb}\rangle_{g_0^k}dV_0 \\
=& \int_{X_0}\langle G_0 \dd_0\lb \phi_i\otimes s_{_\ga}\rb, \bar\pa_0 s_{_\gb}\rangle_{g_0^k}dV_0=0,
\end{split}
\end{align}
because $s_{_\gb}$ is holomorphic. Similarly, we have $\frac{\pa h_{\ga\bar\gb}}{\pa\bar t_j}\bigg |_{t=0}=0$ and
\begin{align}\label{curv10}
\begin{split}
\frac{\pa^2 h_{\ga\bar\gb}}{\pa t_i \pa\bar t_j}\bigg |_{t=0}= & \langle \bar\pa_0^* G_0 \dd_0\lb \phi_i\otimes s_{_\ga}\rb, \bar\pa_0^* G_0 \dd_0\lb \phi_j\otimes s_{_\gb}\rb\rangle_{L^2} \\
& +\int_{X_0}\langle s_{_\ga}, s_{_\gb}\rangle_{g_0^k} \lb \lb k+1\rb \lb \lb \Delta_0+1\rb^{-1}\lb \phi_i\cdot\bar\phi_j\rb+\nu_{i\bar j}\rb -\lb \phi_i\cdot\bar\phi_j\rb\rb dV_0.
\end{split}
\end{align}

Now we analyze the first term on the right hand side of the above formula. Note that, by the proof of Theorem \ref{holext}, we have
\begin{align*}
\begin{split}
\bar\pa_0\dd_0\lb\phi_i\otimes s_{_\ga}\rb=& \bar\pa_0 \lb\phi_i\lrcorner\nabla_0 s_{_\ga}\rb\\
= & \bar\pa_0\phi_i\lrcorner\nabla_0 s_{_\ga}-2k\sqrt{-1}\lb\phi_i\lrcorner\omega_0\rb\otimes s_{_\ga}=0.
\end{split}
\end{align*}
It follows that
\begin{eqnarray}\label{curv20}
\bar\pa_0 G_0\dd_0\lb\phi_i\otimes s_{_\ga}\rb=0.
\end{eqnarray}
Integrating by parts, we get
\begin{eqnarray*}
\langle \bar\pa_0^* G_0 \dd_0\lb \phi_i\otimes s_{_\ga}\rb, \bar\pa_0^* G_0 \dd_0\lb \phi_j\otimes s_{_\gb}\rb\rangle_{L^2}= \langle \dd_0^* G_0 \dd_0\lb \phi_i\otimes s_{_\ga}\rb, \phi_j\otimes s_{_\gb}\rangle_{L^2}.
\end{eqnarray*}
By using equation \eqref{curv20} and the fact that $\bar\pa_0\lb \phi_i\otimes s_{_\ga}\rb=0$, a simple computation shows that
\begin{align*}
\begin{split}
\dd_0^* G_0 \dd_0\lb \phi_i\otimes s_{_\ga}\rb =& \lb \Box_0+k+1\rb^{-1}\lb \dd_0^*\dd_0\lb \phi_i\otimes s_{_\ga}\rb\rb\\
=& \lb \Box_0+k+1\rb^{-1}\Box_0\lb \phi_i\otimes s_{_\ga}\rb\\
=& \phi_i\otimes s_{_\ga}-\lb k+1\rb \lb \Box_0+k+1\rb^{-1}\lb \phi_i\otimes s_{_\ga}\rb.
\end{split}
\end{align*}
Thus
\begin{align*}
\begin{split}
\langle \bar\pa_0^* & G_0 \dd_0\lb \phi_i\otimes s_{_\ga}\rb, \bar\pa_0^* G_0 \dd_0\lb \phi_j\otimes s_{_\gb}\rb\rangle_{L^2}\\
= & \langle \phi_i\otimes s_{_\ga}, \phi_j\otimes s_{_\gb}\rangle_{L^2}-
\lb k+1\rb \langle \lb \Box_0+k+1\rb^{-1}\lb \phi_i\otimes s_{_\ga}\rb, \phi_j\otimes s_{_\gb}\rangle_{L^2}.
\end{split}
\end{align*}
Inserting this into equation \eqref{curv10}, we get
\begin{align}\label{curv30}
\begin{split}
\frac{\pa^2 h_{\ga\bar\gb}}{\pa t_i \pa\bar t_j}\bigg |_{t=0}= & \lb k+1\rb \int_{X_0}\langle s_{_\ga}, s_{_\gb}\rangle_{g_0^k}\lb \lb \Delta_0+1\rb^{-1}\lb \phi_i\cdot\bar\phi_j\rb+\nu_{i\bar j}\rb dV_0\\
&-\lb k+1\rb \int_{X_0}\langle \lb \Box_0+k+1\rb^{-1}\lb \phi_i\otimes s_{_\ga}\rb, \phi_j\otimes s_{_\gb}\rangle_{g_0} dV_0.
\end{split}
\end{align}
The curvature formula \eqref{quanmain10} of the metric $H_k\lb V\rb$ now follows easily from the above formula and equation \eqref{firstord}.

To estimate the limit of Ricci curvatures, we take any vector $v\in T_0^{1,0}B$. By rotation and scaling, we can assume $v=\frac{\pa}{\pa t_1}$. Let $\lfb s_{_\ga}\rfb\subset H^0\lb X_0, K_{X_0}^{-k}\rb$ be an orthonormal basis with respect to the $L^2$ metric. By formula \eqref{quanmain10}, we have
\begin{align}\label{curv40}
\begin{split}
\frac{1}{k+1}\rc_k\lb v,v\rb= & \sum_\ga \int_{X_0} \langle \lb \Box_0+k+1\rb^{-1}\lb \phi_1\otimes s_{_\ga}\rb, \phi_1\otimes s_{_\ga}\rangle_{g_0} dV_0\\
& - \int_{X_0}\tau_k\lb\lb \Delta_0+1\rb^{-1}\lb |\phi_1|^2\rb+\nu_{1\bar 1}\rb dV_0,
\end{split}
\end{align}
where $\tau_k=\sum_\ga \Vert s_{_\ga}\Vert_{g_0^k}^2$ is the Bergman kernel function. Since the operator $\Box_0+k+1$ is self-adjoint and its first eigenvalue is at least $k+1$, we have
\begin{align*}
\begin{split}
0\leq & \sum_{\ga}\int_{X_0} \langle \lb \Box_0+k+1\rb^{-1}\lb \phi_1\otimes s_{_\ga}\rb, \phi_1\otimes s_{_\ga}\rangle_{g_0} dV_0\\
\leq & \sum_{\ga} \frac{1}{k+1}\int_{X_0} \langle \phi_1\otimes s_{_\ga}, \phi_1\otimes s_{_\ga}\rangle_{g_0} dV_0\\
=& \sum_{\ga}\frac{1}{k+1} \int_{X_0} |\phi_1|^2 \Vert s_{_\ga}\Vert_{g_0^k}^2 dV_0=\frac{1}{k+1} \int_{X_0} \tau_k |\phi_1|^2 dV_0.
\end{split}
\end{align*}
Combining the above inequality with equation \eqref{curv40}, we have
\begin{align}\label{curv50}
\begin{split}
0\leq & \frac{1}{k+1}\rc_k\lb v,v\rb  + \int_{X_0}\tau_k\lb \Delta_0+1\rb^{-1}\lb |\phi_1|^2\rb dV_0 \\
\leq & \frac{1}{k+1} \int_{X_0} \tau_k |\phi_1|^2 dV_0.
\end{split}
\end{align}
By the Bergman kernel expansion
\[
\tau_k=\frac{k^n}{\pi^n}+\frac{nk^{n-1}}{2\pi^n}+O\lb k^{n-2}\rb
\]
and the fact that 
\[\omega_{_{WP}}\lb v,v\rb=\int_{X_0}\lb \Delta_0+1\rb^{-1}\lb |\phi_1|^2\rb dV_0, \] 
we have
\begin{align*}
\begin{split}
&\lim_{k\to\infty}\frac{\pi^n}{k^n}\int_{X_0}\tau_k\lb \lb \Delta_0+1\rb^{-1}\lb |\phi_1|^2\rb+\nu_{1\bar 1}\rb dV_0 \\
=&\int_{X_0}\lb 1+\frac{n}{2k}+O\lb k^{-2}\rb\rb\lb\Delta_0+1\rb^{-1}\lb |\phi_1|^2\rb dV_0 + \lim_{k\to\infty}\frac{\pi^n}{k^n}\int_{X_0}\tau_k\nu_{1\bar 1} dV_0\\
=& \omega_{_{WP}}\lb v,v\rb- \lim_{k\to\infty}\int_{X_0}\lb 1+\frac{n}{2k}+O\lb k^{-2}\rb\rb\Delta_0\nu_{1\bar 1} dV_0\\
=& \omega_{_{WP}}\lb v,v\rb
\end{split}
\end{align*}
and
\[
\lim_{k\to\infty}\frac{\pi^n}{k^n} \lb \frac{1}{k+1} \int_{X_0} \tau_k |\phi_1|^2 dV_0\rb=0.
\]
Thus, \eqref{quanmain20} follows from inequality \eqref{curv50} and the above limits directly.
\end{proof}

\section{Plurisubharmonicity of Energy of Harmonic Maps}

Another application of the deformation of \ke metrics, such as Theorem \ref{kexp}, is the variation of energy of harmonic maps. In \cite{toledo12}, Toledo studied the harmonic maps from hyperbolic Riemann surfaces to a fixed Riemannian manifold $\lb N, h\rb$. For a Riemann surface $\Sigma$, fixing a homotopy class $A$ of continuous maps from $\Sigma$ to $N$ and assuming that the sectional curvature of $N$ is nonpositive, there exist smooth harmonic maps from $\Sigma$ to $N$ in the homotopy class $A$. Although such harmonic maps may not be unique, the energy depends only on the conformal structure of $\Sigma$, thus one obtains an energy function $E$ on the Teichm\"uller space $\mathcal T$ of $\Sigma$. Toledo showed that if one further assumes that the curvature of $N$ is Hermitian nonpositive, then $E$ is a plurisubharmonic function on $\mathcal T$.
Shortly after Toledo's work, Yau pointed out that such construction can be used to study the Teichm\"uller spaces of higher dimensional \ke manifolds and the plurisubharmonicity of the energy functions should hold in these cases. This was carried out in \cite{zhang} back in 2014.

Let $\lb X,\omega\rb$ be a \ka manifold with metric $g$ and let $\lb N,h\rb$ be a Riemannian manifold. To ensure the existence of harmonic maps, we assume that $N$ has nonpositive sectional curvature. A $W^{1,2}$-map $f:X\to N$ is harmonic if it minimizes the energy
\[ E(f)=\int_X \left |\pa f\right |^2\frac{\omega^n}{n!}\] in its homotopy class. In this case, $f$ is indeed smooth and satisfies the Euler-Lagrange equation
\begin{eqnarray}\label{harmp10}
\Delta f^\ga+\Gamma_{\gb\gm}^\ga\lb f\rb\frac{\pa f^\gb}{\pa z_i}\frac{\pa f^\gm}{\pa \bar z_j}g^{i\bar j}=0,
\end{eqnarray}
where $\Gamma_{\gb\gm}^\ga$ is the Christoffell symbol of $h$.
Furthermore, the Hopf differential of $f$ is the section
\[H(f)=\frac{\pa f^\ga}{\pa z_i}\frac{\pa f^\gb}{\pa z_k}h_{\ga\gb} dz_i\otimes dz_k\] of $S^2\Omega^{1,0}X$. The curvature of $\lb N,h\rb$ is Hermitian nonpositive if $R^N\lb u,v,\bar u, \bar v\rb\leq 0$ for each point $p\in N$ and all complex tangent vectors $u,v\in T_p^\C N$. If $f:X\to N$ is harmonic, then by using equation \eqref{harmp10} we have the Siu-Sampson identity
\begin{eqnarray}\label{ssid}
\dd\lb \dd\lb H\lb f\rb\rb\rb=-R_{\ga\gb\gm\gd}^N \frac{\pa f^\ga}{\pa z_i}\frac{\pa f^\gm}{\pa \bar z_j}\frac{\pa f^\gb}{\pa z_k}\frac{\pa f^\gd}{\pa \bar z_l}g^{i\bar j}g^{k\bar l}+\Vert \nabla^{1,0}\bar\pa f\Vert^2.
\end{eqnarray}
The following result was shown in Sampson \cite{sampson84}.

\begin{theorem}\label{ssva}
If the curvature of $\lb N,h\rb$ is Hermitian nonpositive and $f:X\to N$ is a harmonic map, then  $\nabla^{1,0}\bar\pa f=0$ and
\[ R_{\ga\gb\gm\gd}^N \frac{\pa f^\ga}{\pa z_i}\frac{\pa f^\gm}{\pa \bar z_j}\frac{\pa f^\gb}{\pa z_k}\frac{\pa f^\gd}{\pa \bar z_l}g^{i\bar j}g^{k\bar l}=0.\]
\end{theorem}

In view of constructing nontrivial plurisubharmonic functions on the \tei spaces of \ke manifolds by using energy of harmonic maps, the Bochner formula implies that the only interesting case is that when each \ke manifold is of general type.

Let $\lb X_0,\omega_0\rb$ be a \ke manifold of general type, and $\phi_1,\cdots,\phi_m\in \h^{0,1}\lb X_0, T^{1,0}X_0\rb$ be a basis of harmonic Beltrami differentials. We consider the power series $\phi(t)$ as in equation \eqref{power10} which is the solution of the Kuranishi equation \eqref{kueq}. In this section, we give a formal discussion of the plurisubharmonicity of the energy of harmonic maps. The study of nonsmoothness of the Kuranishi space of $X_0$, the existence of smooth family of harmonic maps and the asymptotic behavior of the energy function will be discussed elsewhere since they are of independent interests. Thus we assume the deformation of the complex structure on $X_0$ is unobstructed. Let $\lb \XX, B,\pi\rb$ be the Kuranishi family of $X_0$ as constructed in Section \ref{kudiv}. It was shown in \cite{sundeform1} that, in this case, the Kuranishi gauge is equivalent to the divergence gauge. In particular, we have
\begin{eqnarray}\label{gycom}
\phi(t)\lrcorner\omega_0=0.
\end{eqnarray}
To simplify the notation, we assume $m=1$. The general case follows from the same type of computations. The deformation of \ke metrics in this case was established in \cite{sundeform1}. We let $V_t$ and $\omega_t$ be the volume form and the \ka form of the \ke metric on $X_t$, respectively. Then
\begin{align}\label{gyexp10}
\begin{split}
dV_t=& \lb 1+|t|^2\Delta_0(1-\Delta_0)^{-1} \lb |\phi_1|^2\rb+O \lb |t|^3\rb\rb dV_0,\\
\omega_t=& \omega_0+|t|^2 \lb \frac{\sqrt{-1}}{2}\pa_0\bar\pa_0 \lb (1-\Delta_0)^{-1}|\phi_1|^2\rb\rb+O\lb |t|^3\rb.
\end{split}
\end{align}

Now we let $\lb N,h\rb$ be a Riemannian manifold of nonpositive sectional curvature, $A$ be a homotopy class of maps from $X_0$ to $N$, and $F:\XX\to N$ be a smooth map such that each $f_t=F\mid_{X_t}:X_t\to N$ is a harmonic map in the class $A$. We note that the energy function $E\lb t,\bar t\rb=E\lb f_t\rb$ is independent of the choice of $F$ and is a function on $B$. 

\begin{theorem}\label{pshenergy}
The first variation of $E$ is given by
\begin{eqnarray}\label{1stv}
\frac{\pa E}{\pa t}\bigg |_{t=0}=-\int_{X_0} \Lambda\lb \phi_1\lrcorner H(f_0)\rb dV_0
\end{eqnarray}
and the second variation of $E$ is given by
\begin{align}\label{2stv}
\begin{split}
\frac{\pa^2 E}{\pa t\pa\bar t}\bigg |_{t=0}=& -\int_{X_0}  R^N_{\ga\gb\gm\gd}\pa_i f_0^\ga\pa_{\bar j}f_0^\gm\pa_p f_0^\gb\pa_{\bar q}f_0^\gd
 g^{i\bar j}g^{p\bar q}K \ dV_0 + \int_{X_0}  \Vert \nabla^{1,0}\bar\pa f_0\Vert^2 K\ dV_0\\
& -2 \int_{X_0} g^{i\bar j} R^N_{\ga\gb\gm\gd}\pa_i f_0^\ga\pa_{\bar j}f_0^\gm u^\gb\bar u^{\gd} dV_0  +
2\int_{X_0}\Vert \nabla^{1,0}\bar u-\bar\phi_1\lrcorner\bar\pa f_0\Vert^2 dV_0,
\end{split}
\end{align}
where $u=\frac{\pa f_t}{\pa t}\bigg |_{t=0}\in \Gamma \lb f_0^*T^\C N\rb$ and $K=\lb 1-\Delta_0\rb^{-1}\lb |\phi_1|^2\rb$.

Furthermore, if we assume that the curvature of $\lb N,h\rb$ is Hermitian nonpositive then the second variation of $E$ can be expressed as
\begin{eqnarray}\label{2stvhnp}
\frac{\pa^2 E}{\pa t\pa\bar t}\bigg |_{t=0}=-2 \int_{X_0} g^{i\bar j} R^N_{\ga\gb\gm\gd}\pa_i f_0^\ga\pa_{\bar j}f_0^\gm u^\gb\bar u^{\gd} dV_0  +
2\int_{X_0}\Vert \nabla^{1,0}\bar u-\bar\phi_1\lrcorner\bar\pa f_0\Vert^2 dV_0.
\end{eqnarray}
In particular, in this case, the energy function $E$ is plurisubharmonic on $B$.
\end{theorem}

\begin{proof}

Formulas \eqref{gyexp10} and \eqref{cxuniv} give us complete information about the operators $\pa_t$ and $\bar\pa_t$, as well as the \ke metric on $X_t$. Thus, by using formula \eqref{gycom} and the harmonic map equation \eqref{harmp10}, the first variation formula \eqref{1stv} follows from integration by parts. This also leads to the following expression of the second variation of $E$:
\begin{align}\label{2stv10}
\begin{split}
\frac{\pa^2 E}{\pa t\pa\bar t}\bigg |_{t=0}=& \int_{X_0} h_{\ga\gb}\pa_if_0^\ga\pa_{\bar j}f_0^\gb g^{i\bar j}\Delta_0 K\ dV_0 -
\int_{X_0} h_{\ga\gb}\pa_if_0^\ga\pa_{\bar j}f_0^\gb g^{i\bar q}g^{p\bar j}\pa_p\pa_{\bar q} K\ dV_0\\
& -2 \int_{X_0} g^{i\bar j} R^N_{\ga\gb\gm\gd}\pa_i f_0^\ga\pa_{\bar j}f_0^\gm u^\gb\bar u^{\gd} dV_0  +
2\int_{X_0}\Vert \nabla^{1,0}\bar u-\bar\psi\lrcorner\bar\pa f_0\Vert^2 dV_0.
\end{split}
\end{align}
Formula \eqref{2stv} now follows from \ref{2stv10} by integration by parts. Furthermore, if we assume the curvature of $N$ is Hermitian nonpositive then, by the Siu-Sampson vanishing Theorem \ref{ssva}, the first two terms on the right hand side of the second variation formula \eqref{2stv} vanish, thus we have formula \eqref{2stvhnp}. The plurisubharmonicity of $E$ now follows immediately from the Hermitian nonpositivity of the curvature of $N$.
\end{proof}

\medskip


\end{document}